\documentclass[12pt]{amsart}

\newtheorem{theorem}{Theorem}
\newtheorem{proposition}[theorem]{Proposition}
\newtheorem{lemma}[theorem]{Lemma}
\newtheorem{corollary}[theorem]{Corollary}
\newtheorem*{theorem*}{Theorem}

\usepackage[latin1]{inputenc}

\usepackage{amssymb}

\newcommand{\cF}{{\mathcal F} }

\newcommand{\cO}{{\mathcal O} }

\numberwithin{equation}{section}

\begin{document}

\baselineskip=16pt

\title{Deligne pairing and determinant bundle}

\author[I. Biswas]{Indranil Biswas}

\address{School of Mathematics, Tata Institute of Fundamental
Research, Homi Bhabha Road, Mumbai 400005, India}

\email{indranil@math.tifr.res.in}

\author[G. Schumacher]{Georg Schumacher}

\address{Fachbereich Mathematik und Informatik,
Philipps-Universit\"at Marburg, Lahnberge,
Hans-Meerwein-Strasse, D-35032
Marburg, Germany}

\email{schumac@mathematik.uni-marburg.de}

\author[L. Weng]{Lin Weng}

\address{Graduate School of Mathematics, Kyushu University
Fukuoka, 819-0395, Japan}

\email{weng@math.kyushu-u.ac.jp}



\date{}

\begin{abstract}
Let $X \,\longrightarrow\, S$ be a smooth projective surjective
morphism, where $X$ and $S$ are integral schemes over $\mathbb C$.
Let $L_0\, , L_1\, , \cdots \, , L_{n-1}\, , L_{n}$ be line bundles
over $X$. There is a natural isomorphism
of the Deligne pairing $\langle L_0\, , \cdots\, ,L_{n}\rangle$
with the determinant line bundle ${\rm Det}(\otimes_{i=0}^{n} (L_i-
{\mathcal O}_{X}))$.
\end{abstract}

\maketitle

\section{Introduction}

Let $X \,\longrightarrow\, S$ be a smooth family of complex projective
curves parameterized by an integral scheme $S/\mathbb C$. Let $L_0$ and
$L_1$ be line bundles over $X$. In \cite{De}, P. Deligne associated to
this data a line bundle $\langle L_0\, , L_1\rangle$ over
the parameter space $S$. This
construction is now known as the \textit{Deligne pairing}.
S. Zhang extended the Deligne pairing to
arbitrary relative dimension \cite{Zh}.
Let $S$ and $X$ be integral schemes over $\mathbb C$, and let
\begin{equation}\label{e1}
f\, :\, X \,\longrightarrow\, S
\end{equation}
be a smooth projective surjective morphism. Let $n$ be the dimension of
the fibers of $f$. Take algebraic line bundles $L_0\, , L_1\, , \cdots
\, , L_{n-1}\, , L_{n}$ over $X$. The Deligne pairing, \cite{Zh}, is a
line bundle
$$
\langle L_0\,, \cdots\, ,L_{n}\rangle\, \longrightarrow\, S
$$
(the construction is briefly recalled in Section \ref{sec2}).
The map
$$
\text{Pic}(X)^{n+1}\,\longrightarrow\, \text{Pic}(S)
$$
defined by $(L_0\, , \cdots \, , L_{n})\,\longmapsto\, \langle
L_0\,, \cdots\, ,L_{n}\rangle$ is symmetric, and
it is bilinear with respect to the group structure defined by
the tensor product of line bundles and
dualization; it is also compatible with base change.

The Deligne pairing has turned out to be very useful; see
\cite{PRS}, \cite{PS}, \cite{Ga}, \cite{El}, \cite{Fr}.

Given a locally free coherent sheaf $F$ on $X$, we have a line bundle 
${\rm Det}(F)$ on $S$ (see \cite{KM}, \cite{BGS}). This extends to a 
homomorphism to $\text{Pic}(S)$ from the Grothendieck group of locally 
free coherent sheaves on $X$.

The aim of this note is to announce the following:

\begin{theorem*}
There is a canonical isomorphism
$\langle L_0\, , \cdots\, ,L_{n}\rangle
\,\longrightarrow\, {\rm Det}(\otimes_{i=0}^{n} (L_i-
{\mathcal O}_{X}))$.
\end{theorem*}

A theorem connecting Deligne pairing to dtereminant line bundle
is proved in \cite{PRS} (see \cite[p. 476, Theorem 1(1)]{PRS}).

If each $L_i$, $i\,\in\, [0\, ,n]$, is equipped with a $C^\infty$
hermitian structure $h_i$, then $\langle L_0\, , \cdots\, ,L_{n}\rangle$
inherits a hermitian structure \cite{De}, \cite{Zh}. On the other hand,
the hermitian structures $h_1\, ,\cdots\, ,h_n$, the trivial hermitian
structure on the trivial line bundle ${\mathcal O}_{X}$,
and a relative K\"ahler structure on $X$ together define a
hermitian structure on the determinant bundle ${\rm Det}(\otimes_{i=0}^{n}
(L_i- {\mathcal O}_{X}))$ according to \cite{BGS}, \cite{Qu}. The
curvatures of $\langle L_0\, , \cdots\, ,L_{n}\rangle$ and ${\rm
Det}(\otimes_{i=0}^{n} (L_i- {\mathcal O}_{X}))$ coincide (see Proposition
\ref{thm2}). Finally we observe that the Weil--Petersson metric for
families of canonically polarized varieties can be interpreted as the
curvature form of a certain Deligne pairing.

After an initial version was written by the first two authors, the
referee pointed out that this work was also done by the third author.
We thank the referee for this.

\section{A canonical isomorphism}\label{sec2}

We continue with the notation of the introduction.
Let $L_0\, , L_1\, , \cdots
\, , L_{n-1}\, , L_{n}$ be line bundles over $X$. A local
trivialization of $\langle L_0\, , \cdots\, ,L_{n}\rangle$
over some Zariski open subset $U$ of $S$ is given by fixing a
rational section $l_i$ of $L_i$ for each $i\, \in\, [0\, ,n]$,
such that
the intersection $(\bigcap_{i=0}^n{\rm div}(l_i))\bigcap f^{-1}(U)$
is empty. The generator for $\langle L_0\, , \cdots\,
,L_{n}\rangle\vert_U$ corresponding to $\{l_i\}_{i=0}^n$
is denoted by $\langle l_0,\ldots,l_n\rangle$.

To describe the line bundle $\langle L_0\, , \cdots\, ,L_{n}\rangle$
in terms of these trivializations, we need to give the transition
functions for ordered pairs of such trivializations. Let $g$
be a rational function on $X$, and $i\,\in\, [1\, ,n]$.
Assume that $\bigcap_{j\neq i} {\rm div}(l_j) = \sum_k n_kY_k$
is finite over $S$, and that it has empty intersection with ${\rm
div}(g)$ over
an open subset $U'\, \subset\, U$ (this subset is the intersection
of two open subsets of the above type). Then $\langle
l_0,\ldots,g\, l_i,
\ldots,l_n\rangle$ is another generator. The transition function is
given by the following equation:
\begin{equation}\label{trans}
\langle l_0,\ldots,g\, l_i,\ldots,l_n\rangle = \prod_k {\rm
Norm}_{Y_k/S}(g)^{n_k}\langle l_0,\ldots,l_n\rangle\, .
\end{equation}
It is sufficient to describe this type of transition
functions, because a general transition function is a product of
such functions.

Given a coherent sheaf $V$ on $S$, we have a line bundle ${\rm det}(V)$ on
$S$ (see \cite[Ch.~V, \S~6]{Ko}). For a coherent sheaf $F$ on $X$,
we have a line bundle
\begin{equation}\label{e2}
{\rm Det}(F)\, :=\, {\rm det} f_! F\,=\, \otimes_{i=0}^n {\rm
det}(R^if_* F)^{(-1)^i}
\,\longrightarrow\, S
\end{equation}
\cite{KM}, \cite{BGS}. For coherent sheaf $F_1$ and $F_2$ on $X$,
define
$$
{\rm Det}(F_1+F_2)\, :=\,{\rm Det}(F_1)\otimes {\rm Det}(F_2)
~\text{~and~}~
{\rm Det}(F_1-F_2)\, :=\,{\rm Det}(F_1)\otimes {\rm Det}(F_2)^*\, .
$$

\begin{theorem}\label{thm1}
Let $L_0\, , L_1\, , \cdots \, , L_{n}$ be line bundles over
$X$. Then there is a canonical isomorphism
$$
\varphi\, :\, \langle L_0\, , \cdots\, ,L_{n}\rangle
\,\longrightarrow\, {\rm Det}(\otimes_{i=0}^{n} (L_i-
{\mathcal O}_{X}))\, .
$$
\end{theorem}

We begin with the following observation:

\begin{lemma}\label{gro}
Let $D'$ and $D''$ be effective divisors on $X$ such that the
intersection $Y\,:=\, D'\bigcap
D''$ does not contain any divisor. Then the following equality of
elements of $K(X)$ holds:
\begin{equation}\label{mul}
\cO_X(D'-D'') - \cO_X= \cO_{D'}(D') - \cO_{D''} - \cF\, ,
\end{equation}
where $\cF$ is supported on $Y$.
\end{lemma}
\begin{proof}
This follows immediately from the identity
$$
\cO_X(D'-D'')-\cO_X \hspace{11cm}\strut
$$
$$
= \big(\cO_X(D')-\cO_X\big)\otimes\big(\cO_X(-D'')-
\cO_X\big) +
\big(\cO_X(D')-\cO_X\big) +\big(\cO_X(-D'')-\cO_X\big)
$$
with $\cF= \big(\cO_X(D')-\cO_X\big)\otimes\big(\cO_X(-D'')-
\cO_X\big)$.
\end{proof}

We will describe the determinant line bundle ${\rm
Det}(\otimes_{i=0}^{n} (L_i- {\mathcal O}_{X}))$ in terms of
local trivializations and transition
functions. For that purpose, we construct a covering of the base $S$
by Zariski open
subsets $S'$ over which each line bundle $L_i$, $i\, \in\, [0\, ,n]$,
is given by a divisor $D^+_i-D^-_i$, where both $D^+_i$ and
$D^-_i$ are effective, such that the following conditions hold:
any intersection of $n$ hypersurfaces in the union of all these
divisors is reduced, the intersection is of the expected codimension,
and it is finite over the base.

We note that if $Z$ is an intersection
of $n$ hypersurfaces in the union these divisors, then $f(Z)$ is
contained in a divisor on $S$ because $Z$ is of expected codimension.
This choice of divisors for the line bundles gives
a trivialization of ${\rm Det}(\otimes_{i=0}^{n} (L_i-
{\mathcal O}_{X}))$ over the complement of the union of all
$f(Z_I)$, where $I$ runs over the set of $n$ hypersurfaces in the union
the divisors. This open subset, which will be denoted by $S_0$, is
nonempty because each $f(Z_I)$ is
of codimension at least one. There is a trivialization
\begin{equation}\label{g}
\lambda_1\, \in\, H^0(S_0,\, {\rm Det}(\otimes_{i=0}^{n} (L_i-
{\mathcal O}_{X}))\vert_{S_0})\, .
\end{equation}

We now pick a rational function
\begin{equation}\label{g0}
g\, \in\, H^0(X'\, , {\mathcal O}_{X'})
\end{equation}
on $X$. We assume that the divisor $(g)\,:=\, \text{div}(g)$ can be
included in the above
system of divisors so that the above properties continue to hold for
the enlarged system. The line bundle $L_0$ is given by the
divisor $D^+_0-D^-_0+(g)$. Let
\begin{equation}\label{g1}
\lambda_2\, \in\, H^0(S'_0,\, {\rm Det}(\otimes_{i=0}^{n} (L_i-
{\mathcal O}_{X}))\vert_{S_0})
\end{equation}
be the trivialization obtained by replacing $D^+_0-D^-_0$ with
$D^+_0-D^-_0+(g)$ in the construction of the trivialization
$\lambda_1$ in \eqref{g}. Let
\begin{equation}\label{t}
t\,:=\, \lambda_2\otimes \lambda_1^*\, \in\, H^0(S_0\cap S'_0,
\, {\mathcal O}^*_{S_0\cap S'_0})
\end{equation}
be the transition function. For convenience, $S_0\bigcap S'_0$
will be denoted by $S'$.

\begin{lemma}\label{l-t}
The transition function $t$ in \eqref{t} has the following expression:
$$
t= \prod_{\sigma\in \{+,- \}^n} {\rm
Norm}_{Y_\sigma/S'}(g)^{n_\sigma}\, ,
$$
where $Y_\sigma$ and $n_\sigma$ are defined above, and
$g$ is the function in \eqref{g0}.
\end{lemma}

Theorem \ref{thm1} is proved using Lemma \ref{l-t}. The details will
appear elsewhere.

\section{Some applications of Theorem \ref{thm1}}

Let $f\, :\, X\, \longrightarrow\, S$ be as before. Take 
$n+2$ line bundles $L_0\, ,\cdots\, , L_{n+1}$ on $X$, where
$n\,=\, \dim X -\dim S$.

\begin{corollary}\label{cor1}
The line bundle ${\rm Det}(\otimes_{i=0}^{n+1} (L_i-
{\mathcal O}_{X}))$ on $S$ has a canonical trivialization.
\end{corollary}

\begin{proof}
We know that
$$
\langle L_0\, , L_2\, ,L_3\, , \cdots\, ,L_{n+1}\rangle \otimes \langle 
L_1\, , 
\cdots\, ,L_{n+1}\rangle\,=\, \langle L_0\otimes L_1\, , L_2\, ,\cdots\, 
,L_{n+1}\rangle
$$
\cite{Zh}. Therefore, from Theorem \ref{thm1},
$$
 {\rm Det}((L_0\otimes L_1 -{\mathcal O}_{X})\otimes_{i=2}^{n+1}
(L_i- {\mathcal O}_{X}))= {\rm Det}((L_0 -{\mathcal 
O}_{X})\otimes_{i=2}^{n+1}
(L_i- {\mathcal O}_{X})) \otimes {\rm Det}(\otimes_{i=1}^{n+1} (L_i- 
{\mathcal O}_{X}))\, .
$$
As $(L_0\otimes L_1 -{\mathcal O}_{X}) - (L_0 -{\mathcal
O}_{X}) - (L_0 -{\mathcal O}_{X})\,=\,
L_0\otimes L_1 - L_0 -L_1 +{\mathcal O}_{X}$,
this isomorphism gives a trivialization of
$$
{\rm Det}((L_0\otimes L_1 - L_0 -L_1 + {\mathcal 
O}_{X})\otimes_{i=2}^{n+1} (L_i- {\mathcal O}_{X}))\, .
$$
Since
$$
{\rm Det}((L_0\otimes L_1 - L_0 -L_1 + {\mathcal
O}_{X})\otimes_{i=2}^{n+1} (L_i- {\mathcal O}_{X}))\,=\,
{\rm Det}(\otimes_{i=0}^{n+1} (L_i-
{\mathcal O}_{X}))\, ,
$$
we get a trivialization of
${\rm Det}(\otimes_{i=0}^{n+1} (L_i-{\mathcal O}_{X}))$.
\end{proof}

Fix a relative K\"ahler form $\omega_{X/S}$ on the fibration $f$ in
\eqref{e1}. By definition, for some open covering $\{U_i\}$ of $S$, there
exist K\"ahler forms $\omega_{f^{-1}U_i}$ on ${f^{-1}U_i}$ which induce
the relative real $(1,1)$-form $\omega_{X/S}$ on $f^{-1}U_i$. If $S$ is
singular, we require that a K\"ahler form possesses locally a
$\partial\overline\partial$--potential on some smooth ambient space.

If $F$ is a vector bundle on $X$ equipped with a hermitian structure
$h_F$, then there is a natural hermitian structure on the line bundle
${\rm Det}(F)\,\longrightarrow\, S$
which is constructed using $h_F$ and $\omega_{X/S}$
\cite{Qu}, \cite{BGS}; it is known as the \textit{Quillen metric}. If
$F_1$ and $F_2$ are vector bundles equipped with hermitian structure, then
the hermitian structures on ${\rm Det}(F_1)$ and ${\rm Det}(F_2)^*$
together induce a hermitian structure on ${\rm Det} (F_1-F_2)\, =\, {\rm
Det}(F_1)\otimes {\rm Det}(F_2)^*$.

For each $l\, \in\, [0\, ,n]$, fix a hermitian metric $h_j$ on the line
bundle $L_j$ over $X$. These $h_j$ produce a hermitian metric on $\langle
L_0\, , \cdots\, ,L_{n}\rangle$ \cite{De}, \cite[\S~1.2]{Zh}. Therefore,
both the line bundles ${\rm Det}(\otimes_{i=0}^{n} (L_i- {\mathcal
O}_{X}))$ and $\langle L_0\, , \cdots\, ,L_{n}\rangle$ are equipped with a
hermitian metric.

\begin{proposition}\label{thm2}
The curvature of the hermitian metric on
$\langle L_0\, , \cdots\, ,L_{n}\rangle$ coincides with the
curvature of the Quillen metric on ${\rm Det}(\otimes_{i=0}^{n}
(L_i- {\mathcal O}_{X}))$.
\end{proposition}

\begin{proof}
The Chern form of the metric on $\langle L_0,\ldots,L_n\rangle$ equals the
fiber integral
\begin{equation}\label{del}
\int_{X/S} c_1(L_0,h_0)\wedge \ldots \wedge c_1(L_n,h_n)
\end{equation}
(see \cite{Zh}). On the other hand, a theorem of Bismut, Gillet and
Soul\'e \cite{BGS} says that the Chern form of the determinant line
bundle is the degree two component of the Riemann--Roch fiber
integral
\begin{equation}\label{ht2}
c_1({\rm Det}(\otimes_{i=0}^{n}
(L_i- {\mathcal O}_{X})),h^Q)\,=\,
\Big(\int_{X/S}ch(L_0-\cO_X,h_0)\cdot\ldots\cdot
ch(L_n-\cO_X,h_n)td(X/S)\Big)_{(2)}\, ,
\end{equation}
where $h^Q$ is the Quillen metric on ${\rm Det}(\otimes_{i=0}^{n}
(L_i- {\mathcal O}_{X}))$ (this theorem of \cite{BGS}  
was extended to (smooth) K\"ahler fibrations
over singular base spaces in \cite[\S 12]{f-s}).

Note that
$$
ch(L-\cO_X) \,=\, c_1(L,h) + ~\text{~higher~order~terms}\, .
$$
Hence the only contribution of $td(X/S)$ in \eqref{ht2}
is the constant one, and also the higher order terms in
$ch(L-\cO_X)$ do not contribute.
Consequently, \eqref{ht2} coincides with \eqref{del}.
\end{proof}

Let $X \,\longrightarrow\, S$ be a projective family of canonically
polarized varieties.
Equip the relative canonical bundle $K_{X/S}$ with the hermitian metric
that is induced by the fiberwise K\"ahler-Einstein metrics.
It was shown in \cite{f-s} that the generalized Weil-Petersson
form is equal, up to a
numerical factor, to the fiber integral
$$
\omega_{WP}\simeq \int_{X/S} c_1(K_{X/S},h)^{n+1}\, .
$$
Therefore, we have the following:

\begin{proposition}
Let $X \,\longrightarrow\, S$ be a projective family of canonically
polarized varieties.
The curvature of the metric on the Deligne pairing
$\langle K_{X/S},\ldots,K_{X/S}\rangle$ given by the
fiberwise K\"ahler-Einstein metric coincides with the generalized
Weil-Petersson form $\omega_{WP}$ on $S$.
\end{proposition}


\end{document}